\documentclass{amsart}
\usepackage{mathrsfs}
\usepackage{amsfonts,amssymb,amsmath,amsthm}
\usepackage{url}
\usepackage{enumerate}
\usepackage{float}
\usepackage{subfig, epsfig}
\usepackage{graphicx}

\usepackage{hyperref}
\hypersetup{
    colorlinks=true,
    linkcolor=blue,
    filecolor=cyan,
    urlcolor=blue,
    citecolor=green,
}

\newtheorem{theorem}{Theorem}[section]
\newtheorem{lemma}[theorem]{Lemma}
\newtheorem{proposition}[theorem]{Proposition}
\newtheorem{corollary}[theorem]{Corollary}

\newtheorem{remark}[theorem]{Remark}

\numberwithin{equation}{section}

\allowdisplaybreaks[4]

\author{Yunlong Yang}
\address{
School of Science\\
Dalian Maritime University\\
Dalian, 116026,  P. R. China}
\email{lnuyylong425@163.com}

\author{Yanwen Zhao}
\address{
School of Science \\
Dalian Maritime University\\
Dalian, 116026,  P. R. China }
\email{1789093672@qq.com}

\author{Jianbo Fang}
\address{School of Mathematics and Statistics\\ Guizhou University of Finance and Economics\\
Guiyang, 550025,  P. R. China }
\email{fjbwcj@126.com}

\author{Yanlong Zhang}
\address{Institute of Mathematics\\ 
Henan Academy of Sciences\\
Zhengzhou, 450046,  P. R. China}
\email{ylzhang@hnas.ac.cn}



\thanks{This work is supported by the Fundamental Research Funds for the Central Universities (Nos.3132023202, 3132022206).}

\begin{document}

\title[Inverse flow of planar curves with singularities]
{On length-preserving and area-preserving inverse curvature flow of planar curves with singularities}

\begin{abstract}
This paper aims to investigate the
evolution problem for planar curves with singularities. 
Motivated by the inverse curvature flow introduced by
Li and Wang (Calc. Var. Partial Differ. Equ. 62 (2023), No. 135), we intend to consider 
the area-preserving and length-preserving inverse curvature flow with nonlocal term for $\ell$-convex Legendre curves. For the area-preserving flow, an $\ell$-convex Legendre curve 
with initial algebraic area $A_0>0$ evolves
to a circle of radius $\sqrt{\frac{A_0}{\pi}}$.
For the length-preserving flow, an $\ell$-convex Legendre curve 
with initial algebraic length $L_0$ evolves to a circle of radius $\frac{L_0}{2\pi}$.
As the by-product, we obtain some geometric
inequalities for $\ell$-convex Legendre curves
through the length-preserving flow.
\end{abstract}

\maketitle

\section{Introduction}\label{sec1}

The evolution problem for curves is a fundamental problem in geometry and topology, which has gained much attention in pure mathematics and has been widely applied in fields such as computer vision, image processing, and material science. 
One of the most well-known models may be the curve shortening flow of planar curves which is equivalent to a nonlinear parabolic equation for curvature. Normally, 
the Frenet frame cannot be built at singular points, so it is quite difficult to 
define the evolution problem for curves with singularities even in the plane.

The aim of this paper is to study evolution problems for some planar curves with singularities. Before introducing the models of this paper, we first illustrate basic definitions and notation about curves with singularities in order to comprehend associated evolution problems.

\subsection{Legendre curves}\label{sec1.1}

The curve $\gamma: I\rightarrow{\mathbb{R}^2} $ is referred to an {\it Legendre curve}, if there exists a unit vector
field $\nu: I\rightarrow \mathbb{S}^1$ satisfying
\begin{equation}\label{l-con}
    \langle\gamma'(\theta),\nu(\theta)\rangle=0
\end{equation}
for any $\theta\in I$.
Here, $\gamma'(\theta)=\frac{d\gamma}{d\theta}$ and $\langle\cdot,\cdot\rangle$ denotes the inner product.
Notably, when the vector field $\nu$ is smooth, this special class of curves is termed {\it frontal} as introduced in \cite{F-T2013}.
This definition allows 
us to establish the Frenet frame even at singular points.

Let $\gamma$ be a frontal.  
In this context, the pair $({\nu(\theta),\mu(\theta)})$, satisfying $\mu(\theta) = J(\nu(\theta))$, constitutes a moving frame along the curve $\gamma(\theta)$ in $\mathbb{R}^2$,  where $J$ represents a counterclockwise rotation of $\frac{\pi}{2}$ on $\mathbb{R}^2$. The Frenet formula for $\gamma$ is defined as:
$$\left(
\begin{array}{c}
     \nu'(\theta)  \\
      \mu'(\theta)
\end{array}
\right)=\left(
\begin{array}{cc}
    0 & \ell(\theta) \\
    -\ell(\theta) & 0 
\end{array}
\right)\left(
\begin{array}{c}
     \nu(\theta)  \\
      \mu(\theta)
\end{array}
\right),$$
Here, $\ell(\theta) = \langle \nu'(\theta),\mu(\theta)\rangle$, and there exists a smooth function $\beta$ such that
$$\gamma'(\theta)=\beta(\theta)\mu(\theta).$$
Consequently, a Legendre curve $\gamma$ is regular if and only if $\beta$ never vanishes. The pair $(\ell, \beta)$ serves as a crucial invariant for Legendre curves, commonly referred to as the {\it curvature pair} of Legendre curves, as discussed in \cite{F-T2016}. A straightforward computation reveals the relationship between the conventional curvature $\kappa$ and $\beta$ for Legendre curves, specifically, $\kappa = \frac{\ell}{|\beta|}$.
Also, it is easy to see that the singular points on a Legendre curve satisfy $\beta = 0$. Furthermore, if both $\ell$ and $\beta$ maintain consistent sign, the Legendre curve is convex, and vice versa, as detailed in \cite{F-T2016}. In essence, Legendre curves can be seen as an extension of the 
convex curves.
When $\ell > 0$, an Legendre curve is denoted as an {\it $\ell$-convex Legendre curve}. Remarkably, as established in \cite[Lemma 3.2]{L-W2023}, a Legendre curve can be transformed into one with $\ell = 1$ through reparametrization. Therefore, for the sake of simplicity, we may exclusively focus on Legendre curves with $\ell = 1$.

Let $\gamma:\mathbb{S}^1\rightarrow\mathbb{R}^2$ be an $\ell$-convex Legendre curve with $\ell=1$. According to \cite[Lemma 3.3]{L-W2023}, the curve $\gamma$ can be expressed as
\begin{equation}\label{qx}
\gamma(\theta) = p(\theta)(\cos \theta, \sin \theta) + p'(\theta)(-\sin \theta, \cos \theta),
\end{equation}
Here, $(\cos\theta,\sin\theta)$ is denoted as the unit vector field $\nu$ on $\mathbb{S}^1$. The function $p(\theta)$ serves the same role as the support function for a convex curve and is still called the {\it support function} for the $\ell$-convex Legendre curve $\gamma$ in
\cite{L-W2023}.
Notably, $p(\theta)>0$ and $p(\theta)+p''(\theta)>0$ for any $\theta\in\mathbb{S}^1$, curve $\gamma$ turns into a convex curve. Moreover, in line with \cite[Remark 3.5]{L-W2023}, we have
\begin{equation}\label{ql}
\beta(\theta) = p(\theta) + p''(\theta).
\end{equation}
Much like convex curves, the algebraic length $L$ and the algebraic area $A$ of $\gamma$ can be defined, as detailed in \cite[Definition 3.4]{L-W2023}, by the following expressions
\begin{align}
    &L=\int_{\mathbb{S}^1}(p(\theta)+p''(\theta))d\theta=\int_{\mathbb{S}^1}p(\theta)d\theta,\label{L}\\
    &A=\frac{1}{2}\int_{\mathbb{S}^1}p(\theta)(p(\theta)+p''(\theta))d\theta.\label{A}
\end{align}
It's worth noting that replacing $p$ with $-p$ only alters the sign of $L$ while keeping the sign of $A$. Therefore, we may pay attention to $\ell$-convex curves with 
$L\ge 0$.
Additionally, it's important to highlight that the algebraic area $A$ of a Legendre curve may be positive, zero, or even negative values.

If we consider the $\ell$-convex curve $\gamma$ with a support function expressed as
\begin{equation}\label{exp1}
    p(\theta)=a_0+\sum_{k\ge 1}(a_k\cos k\theta+b_k \sin k\theta),
\end{equation}
then, through 
integration by parts and Parseval's identity, we obtain
\begin{align}
    &L=2\pi a_0,\label{L1}\\
    &A=\pi a_0^2+\frac{\pi}{2}\sum_{k\ge 2}(1-k^2)(a_k^2+b_k^2),\label{A1}
\end{align}
see e.g. \cite{L-W2023}.
Additionally, the coefficients $a_1$ and $b_1$
appeared in \eqref{exp1} own great significance in determining the position of the curve $\gamma$.
Concretely, $$a_1=\frac{1}{\pi}\int_0^{2\pi}p(\theta)\cos\theta d\theta$$
and $$b_1=\frac{1}{\pi}\int_0^{2\pi}p(\theta)\sin\theta d\theta.$$
The point $(a_1,b_1)$  is commonly referred to be the {\it Steiner point} for convex curves, see \cite[p.50]{S2014}.

Due to the isoperimetric inequality for $\ell$-convex Legendre curves established in \cite[Theorem 4.1]{L-W2023}, if its algebraic length $L$ of such a curve is equal to zero, then its algebraic area $A$ must be a negative value, unless the curve degenerates into a single point. However, when $L$ is greater than zero, the situation for $A$ becomes notably intricate. For instance, consider the support function of $\gamma$ in the form of $p(\theta) = 2 + \sin 2\theta$.
In this case, the algebraic area $A$ equals to $\frac{5\pi}{2}$.
Alternatively, if $p(\theta) = \sqrt{\frac{3}{2}} + \sin 2\theta$, the corresponding $A$ evaluates to zero.
Lastly, for the support function $p(\theta) = \frac{1}{2} + \sin 2\theta$, the associated algebraic area $A$ becomes $-\frac{5\pi}{4}$; see Figure \ref{fig1}.

\begin{figure}
\begin{center}
\subfloat[
    $p(\theta)=2+\sin 2\theta$
\label{fig1-2}]{\includegraphics[width=0.4\textwidth]{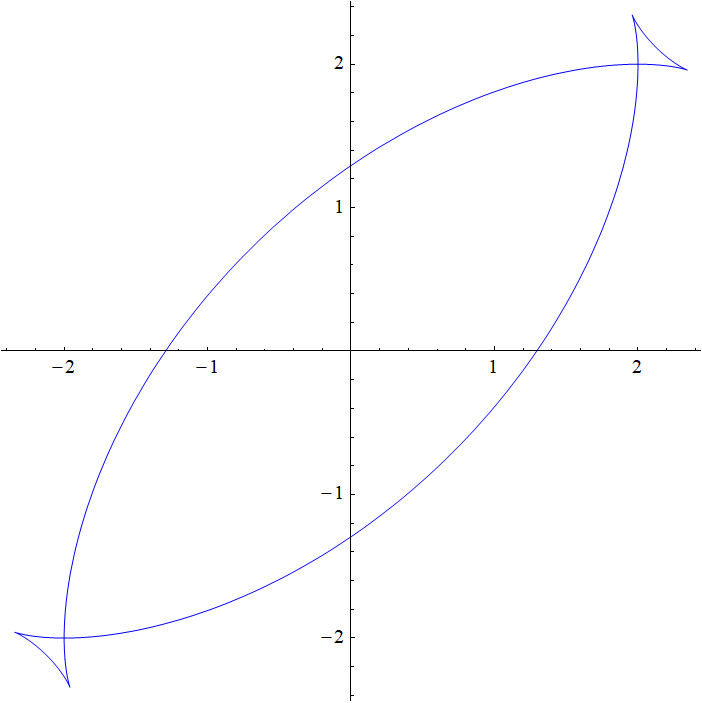}}\qquad
\subfloat[$p(\theta)=\sqrt{\frac{3}{2}}+\sin 2\theta$
\label{fig1-3}]{\includegraphics[width=0.4\textwidth]{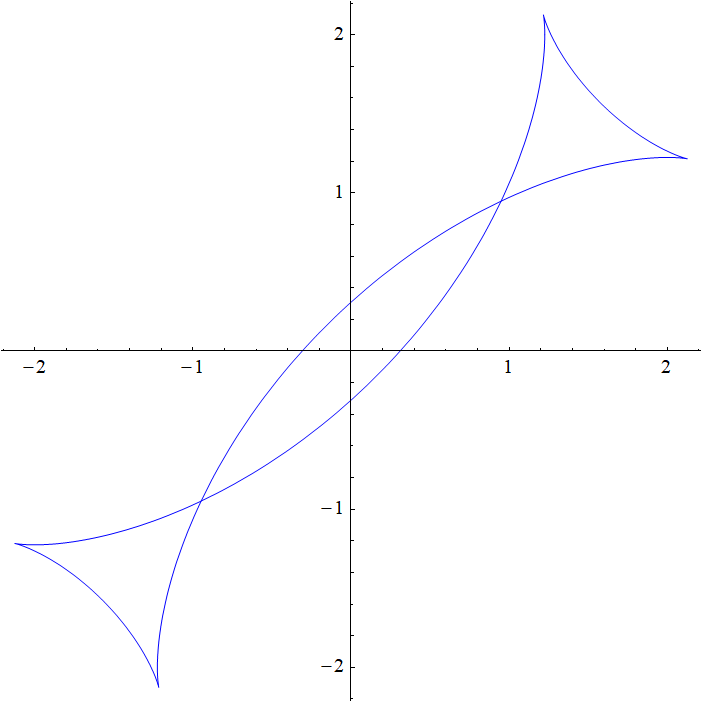}}\qquad
\subfloat[$p(\theta)=\frac{1}{2}+\sin 2\theta$
\label{fig1-4}]{\includegraphics[width=0.4\textwidth]{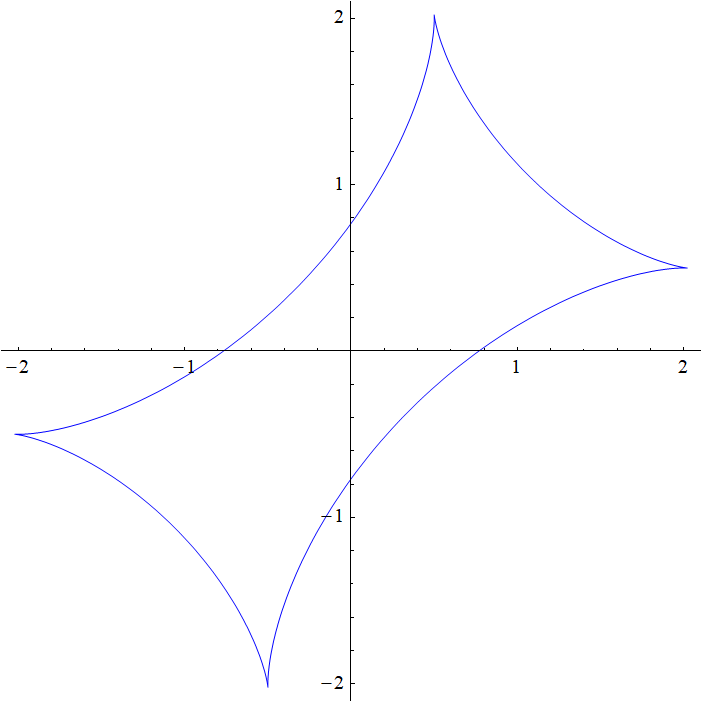}}\qquad
\subfloat[
    $p(\theta)=2\sin 2\theta$
\label{fig1-1}]{\includegraphics[width=0.4\textwidth]{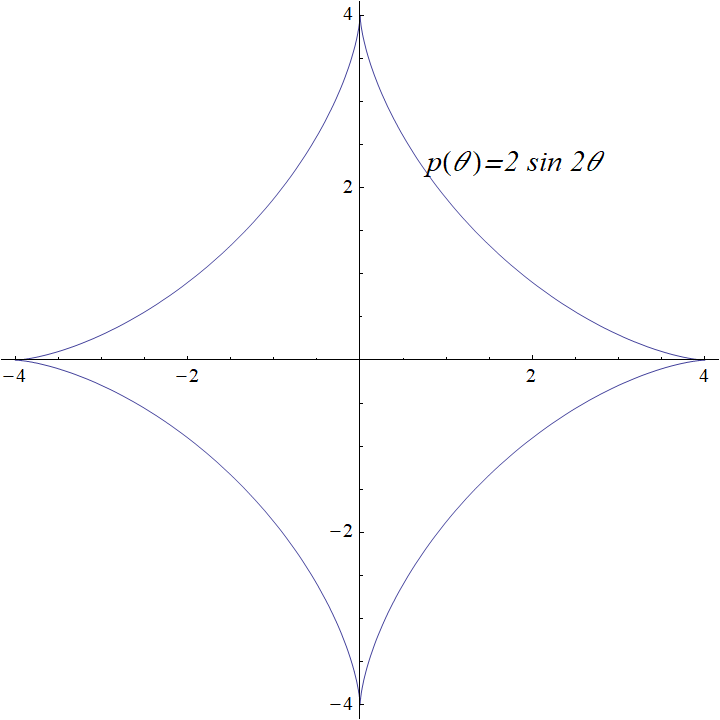}}
\caption{Examples
\label{fig1}}
\end{center}
\end{figure}


\subsection{Evolution problems and main theorems}

Inspired by the impressive inverse curvature flow researched by Li and Wang \cite{L-W2023},
we intend to consider the evolution problem
for $\ell$-convex Legendre curves as follows
\begin{equation}\label{model}
    \begin{cases}
        \frac{\partial \gamma}{\partial t}=\frac{\ell (\beta-\lambda(t))^{\prime}}{\ell^{2}+\beta^{2}} \mu+(\beta-\lambda(t)) \nu, \\
        \frac{\partial \nu}{\partial t}=-\frac{\beta (\beta-\lambda(t))^{\prime}}{\ell^{2}+\beta^{2}} \mu,
    \end{cases}
\end{equation}
where $\lambda(t)$ is a nonlocal term.


The first purpose of this paper is to study nonlocal
flows with $\lambda(t)=\frac{1}{L}\int_0^{2\pi}\beta^2 d\theta$
or $\lambda(t)=\frac{L}{2\pi}$.
From the evolution equations for $L$ and $A$ (see Lemma \ref{lem2.1}), the flow \eqref{model} is area-preserving 
when $\lambda(t)=\frac{1}{L}\int_0^{2\pi}\beta^2 d\theta$, and it is length-preserving when $\lambda(t)=\frac{L}{2\pi}$.

For the area-preserving flow, one can get 

\begin{theorem}\label{thm1.1}
Let $\gamma_0$ be an initial $\ell$-convex Legendre curve with $\ell(\theta,0)=1$. If
the algebraic area of $\gamma_0$ is positive, then for $\lambda(t)=\frac{1}{L}\int_0^{2\pi}\beta^2 d\theta$, the flow \eqref{model} exists in the time interval $[0,+\infty)$. Meanwhile, the evolving curve is still $\ell$-convex, preserves the algebraic area, and converges
to a circle of radius $\sqrt{\frac{A_0}{\pi}}$
in the $C^{\infty}$ sense as time $t$ goes to infinity.
\end{theorem}

\begin{remark}
The initial algebraic area $A_0>0$, as mentioned in the area-preserving flow of Theorem \ref{thm1.1}, stands as an indispensable prerequisite for the evolving  $\ell$-convex Legendre curves. In fact, if $A_0\le 0$, there is the possibility that $L_0=0$. This makes nonlocal term $\lambda(t)=\frac{1}{L}\int_0^{2\pi}\beta^2 d\theta$ no sense.
For example, consider that the support function of $\gamma_0$ is $p_0(\theta)=\sin\theta+2\cos\theta$.
In this case, the algebraic area $A_0=0$.
And if the support function $p_0(\theta)=\sin\theta+2\cos\theta+\sin2\theta+2\cos2\theta$, its algebraic area $A_0<0$. In both situations, the associated algebraic length $L_0=0$.
\end{remark}

For the length-preserving flow, one has 

\begin{theorem}\label{thm1.2}
Let $\gamma_0$ be an initial $\ell$-convex Legendre curve with $\ell(\theta,0)=1$. For $\lambda(t)=\frac{L}{2\pi}$, the flow \eqref{model} exists in the time interval $[0,+\infty)$. Meanwhile, the evolving curve is still $\ell$-convex, preserves the algebraic length and converges
to a circle of radius $\frac{L_0}{2\pi}$.
Specially, if $L_0=0$, then it converges to a point. 
\end{theorem}

In essence, the flow \eqref{model} is an inverse flow for $\ell$-convex Legendre curves. In particular, when the initial curve becomes a convex curve, this flow represents Gao-Pan-Tsai's area-preserving model \cite{G-P-T2021} when $\lambda(t)=\frac{1}{L}\int_0^{2\pi}\beta^2 d\theta$, and corresponds to Pan-Yang's length-preserving flow \cite{P-Y2008} when $\lambda(t)=\frac{L}{2\pi}$. Unlike previous work on the inverse curvature flow for convex curves, see \cite{G-P-T2020, G-P-T2021}, the long-term existence of the length-preserving and area-preserving flow for $\ell$-convex Legendre curves does not necessitate additional curvature conditions.
More insights into the inverse curvature flow for convex curves can be found in \cite{K2019}, and the related literature therein.

The exploration of evolution problems for planar convex curves began with the comprehensive work of Gage and Hamilton \cite{G-H1986}. 
They established the renowned result that 
a convex curve evolves into a point at
finite time under the curve shortening
flow, and the normalized of evolving curve converges to a circle as time goes to infinity. 
As a natural extension of this research, Gage \cite{G1986} introduced a nonlocal area-preserving  flow. The area-preserving flows with nonlocal speed for planar convex curves have since garnered substantial attention and undergone extensive investigations \cite{G-P-T2021, M-C2014, M-P-W2013, P-Y2019, T-W2015, Y-Z-Z2023}. 
Distinct from Gage's area-preserving model, Pan and Yang \cite{P-Y2008} introduced 
an inverse curvature flow  that is a length-preserving flow with nonlocal term. Further investigations into length-preserving flows have been undertaken in \cite{G-P-T2020, T-W2015}.
For immersed curves, corresponding nonlocal area-preserving and length-preserving flows have been discussed by \cite{S-T-W2020,W2023, W-W-Y2018}. 
The nonlocal flows for general polygons are discussed in \cite{G-G-K-O2022, P-Z2020}.
In another significant extension, Grayson \cite{G1987} demonstrated that any embedded planar curve can evolve into a convex one within finite time. This extends the Gage-Hamilton's result to embedded curves and is known as the Gage-Hamilton-Grayson theorem.
This theorem has also been studied in the context of nonlocal area-preserving and length-preserving flows, as detailed in \cite{D2021, N-N2019, N2021, G-P2023}. 
There is another nonlocal flow that is a gradient flow for the isoperimetirc ratio of evolving curve
\cite{J-P2008}. 
Research on
the associated parabolic equations for evolving problems can be found in works such as \cite{A1988, A1990, A1991, G-K1985}.
For a comprehensive understanding various aspects of evolving problems, readers can refer to the monographs \cite{A-C-G-L2020} and \cite{C-Z2001}.


For an $\ell$-convex Legendre curve $\gamma$ with algebraic length
$L$ and algebraic area $A$, 
Li and Wang \cite{L-W2023} have established a set of geometric inequalities: 
\begin{align}
&L^2-4\pi A\ge 0,\label{in-0}\\
    &\int_0^{2\pi}\beta^2 d\theta-2A\ge 0,\label{in-1}\\
    &\int_0^{2\pi}\beta^2 d\theta-2A-8\left(\frac{L^2}{4\pi}-A\right)\ge 0,\label{in-2}\\
    &\frac{1}{12}\int_0^{2\pi}\beta'^2 d\theta-2\left(\frac{L^2}{4\pi}-A\right)\ge 0.\label{in-3}
\end{align}
Inequality \eqref{in-0} constitutes the isoperimetric inequality for $\ell$-convex Legendre curves.
Differing from the classical isoperimetric inequality, it remains valid even for non-simple curves. Notably, the equalities in \eqref{in-0} and \eqref{in-1} hold if and only if $\gamma$ is a circle, while the equalities in \eqref{in-2} and \eqref{in-3} hold if and only if the support function of $\gamma$ follows the form
\begin{equation}\label{supas}
p(\theta)=a_0+a_1\cos\theta+b_1\sin\theta+a_2\cos2\theta+b_2\sin2\theta.   
\end{equation}
The proofs of these inequalities rely on
Fourier series. 
For additional geometric inequalities related to convex curves, we refer the readers to \cite{G-O1999,K-L2021} and the references therein. It should be noted that if the support function of an $\ell$-convex Legendre curve 
is form \eqref{supas}, it is a parallel curve of an astroid centered at $(a_1,b_1)$. An example of the astroid can be seen in
Figure \ref{fig1}(D) and the associated introduction to this kind of curve can be found in \cite[Example 3.3]{L-W2023} (see also \cite{F-T2016}).

The second objective of this paper is to derive some geometric inequalities through the exploration of evolution problems concerning $\ell$-convex Legendre curves.
As no geometric inequalities are employed in the discussion of the length-preserving inverse curvature flow presented in Theorem \ref{thm1.2}, this particular model is more advantageous in the pursuit of deriving geometric inequalities.
To be specific, the length-preserving inverse curvature flow in Theorem \ref{thm1.2} allows us to derive the isoperimetric inequality \eqref{in-0} and to obtain the generalizations for inequalities \eqref{in-1}, \eqref{in-2}, and \eqref{in-3}. 


The inverse curvature flow plays a pivotal role in the derivation of geometric inequalities, and there exists a body of research dedicated to this topic, with notable contributions found in \cite{B-H-W2016, G-L2018, H-I2001}, among others. The locally constrained inverse curvature flow, introduced by Brendle et al. \cite{B-G-L2020} (see also Guan and Li \cite{G-L2015}), was designed to ensure the monotonicity of certain geometric quantities and weaken the initial condition in some cases.
The earliest investigations into the inverse curvature flow for hypersurfaces in Euclidean space can be traced back to Gerhardt \cite{G1990} and Urbas \cite{U1990}. Building upon these foundational works, Guan and Li \cite{G-L2009} demonstrated the monotonicity of the ratio of quermassintegrals under the associated inverse curvature flow. They also provided a proof for Alexandrov-Fenchel inequalities applicable to $k$-convex and star-shaped domains.
In recent times, a lot of research efforts have been dedicated to locally constrained inverse curvature flows, as evidenced by works such as \cite{G-L2021, H-L2022, H-L-W2022, K-W-W-W2022, S-X2019, X2017-1, X2017-2}, among others. 
For the latest developments 
focused on the inverse curvature flow for convex curves, readers can explore  in \cite{G-P-T2020, G-P-T2021, K2019, Y-Z-Z2023}, and the literature therein.



This paper is organized as follows. In Section \ref{sec2}, we present some basic concepts and results about geometric flows for $\ell$-convex Legendre curves. In Section \ref{sec4},
we deal with the area-preserving flow. 
In Section \ref{sec3}, we research the length-preserving inverse curvature flow.
In Section \ref{sec5}, we can get some
geometric inequalities for $\ell$-convex Legendre curves through the length-preserving inverse curvature flow.

\section{Some facts and basic lemmas}\label{sec2}

As mentioned in \cite{L-W2023}, in order to investigate the inverse curve flow for $\ell$-convex curves while preserving the Legendrian condition \eqref{l-con}, the pair $(\gamma,\nu)$ must be considered, as opposed to solely $\gamma$. This involves the following equations
\begin{equation}\label{yl}
    \begin{cases}
        \frac{\partial \gamma}{\partial t}=\frac{\ell f^{\prime}}{\ell^{2}+\beta^{2}} \mu+f \nu, \\
        \frac{\partial \nu}{\partial t}=-\frac{\beta f^{\prime}}{\ell^{2}+\beta^{2}} \mu,
    \end{cases}
\end{equation}
where $f$ is a smooth function about the evolving Legendre curve $\gamma$. 
It is worth noting that, since altering the tangential vector field does not affect the flow, for the sake of simplifying the associated analysis, one can introduce an appropriate tangential vector field $(\frac{\beta^2 f'}{\ell(\ell^2+\beta^2)}\mu, \frac{\beta f'}{(\ell^2+\beta^2)}\mu)$ (see \cite[p.20]{L-W2023}) such that
\begin{equation}\label{jh}
    \begin{cases}
        \frac{\partial \gamma}{\partial t}=\frac{f'}{\ell} \mu+f \nu, \\
        \frac{\partial \nu}{\partial t}=0.
    \end{cases}
\end{equation}
Following some straightforward calculations, one can deduce
\begin{align}
    &\frac{\partial \mu}{\partial t}=0,\label{eqn2.10}\\ 
    &\frac{\partial \ell}{\partial t}=0,\label{eqn2.11}\\
    &\frac{\partial \beta}{\partial t}=\left(\frac{f'}{\ell}\right)'+f\ell.\label{eqn2.12}
\end{align}
As a direct consequence of \eqref{eqn2.11}, it follows that $\ell(\theta,t)=1$ if the initial curve is with $\ell(\theta,0)=1$. Without loss of generality, we can assume $\ell(\theta,0)=1$.
Based on the discussions above and through elementary computations, similar to those outlined in \cite[pp.133-135]{L-W2023}, we can derive the evolution equations for essential geometric quantities.

\begin{lemma}\label{lem2.1}
 The evolution equations for the support function $p$, algebraic length $L$, algebraic area $A$ and the quantity $\beta$ are
\begin{align}
    &\frac{\partial p}{\partial t}=f,\label{eqn2.6}\\
     &\frac{d L}{d t}=\int_0^{2\pi}f d\theta,\label{eqn2.7}\\
     &\frac{d A}{d t}=\int_0^{2\pi}\beta fd\theta,\label{eqn2.8}\\
     &\frac{\partial \beta}{\partial t}=f_{\theta\theta}+f.\label{eqn2.9}
\end{align}   
\end{lemma}

\begin{lemma}\label{lem2.2}
Under flow \eqref{jh} when $f=\beta-\lambda(t)$
and $\lambda(t)$ is a nonlocal term, the evolving curve is always $\ell$-convex.
\end{lemma}

\begin{proof}
In order to show that $\gamma(\theta,t)$ is still $\ell$-convex, by \cite[Remark 3.5]{L-W2023}, we only need to check 
\begin{equation}\label{btj}
   \int_{0}^{2\pi}\beta(\theta,t)\cos\theta d\theta= \int_{0}^{2\pi}\beta(\theta,t)\sin\theta d\theta=0.
\end{equation} 
By \eqref{eqn2.9}, one has
\begin{align*}
\frac{d}{ dt}\int_{0}^{2\pi }  \beta \cos \theta d\theta
&=\int_{0}^{2\pi }\beta _{t}\cos  \theta d\theta \\
&=\int_{0}^{2\pi }\left ( \beta_{\theta\theta}+\beta -\lambda(t)   \right ) \cos \theta d\theta\\
&=-\int_{0}^{2\pi }\beta\cos  \theta d\theta +\int_{0}^{2\pi }\beta \cos  \theta d\theta-\lambda(t)\int_{0}^{2\pi }\cos \theta d\theta
=0.
\end{align*}
Similarly, $\frac{d}{ dt}\int_{0}^{2\pi }  \beta \sin \theta d\theta=0$. Again from the fact that $\gamma_0$ is $\ell$-convex, it yields
\begin{align*}
    &\int_{0}^{2\pi}\beta(\theta,t)\cos\theta d\theta=\int_{0}^{2\pi}\beta(\theta,0)\cos\theta d\theta=0,\\
    &\int_{0}^{2\pi}\beta(\theta,t)\sin\theta d\theta=\int_{0}^{2\pi}\beta(\theta,0)\sin\theta d\theta=0,
\end{align*}
which concludes the desired result.
\end{proof}

\begin{corollary}\label{lem-gd}
Under flow \eqref{jh} when $f=\beta-\lambda(t)$
and $\lambda(t)$ is a nonlocal term, then
\begin{equation}\label{gdds}
\int_{0}^{2\pi }  \beta^{(i)} \cos \theta d\theta=\int_{0}^{2\pi }  \beta^{(i)} \sin \theta d\theta=0    
\end{equation}
holds for $i\ge 1$, where $\beta^{(i)}$ represents the $i$-th derivative
of $\beta$.
\end{corollary}

\begin{proof}
For $i=1$, from integration by parts and \eqref{btj}, it yields
\begin{align*}
&\int_{0}^{2\pi }\beta _{\theta}\cos  \theta d\theta= \int_{0}^{2\pi }\beta\sin  \theta d\theta=0,\\
&\int_{0}^{2\pi }\beta _{\theta}\sin  \theta d\theta
=-\int_{0}^{2\pi }\beta\cos  \theta d\theta=0.
\end{align*}
Assume that the assertion of \eqref{gdds} holds
for $i-1$. By the induction hypothesis, one can compute
\begin{align*}
  \int_{0}^{2\pi }\beta^{(i)}\cos  \theta d\theta= \int_{0}^{2\pi }\beta^{(i-1)}\sin  \theta d\theta=0.  
\end{align*}
In the same way, $\int_{0}^{2\pi }\beta^{(i)}\sin  \theta d\theta=0$.
\end{proof}

\begin{lemma}\label{lem2.3}
Under flow \eqref{jh} when $f=\beta-\lambda(t)$,
if $\lambda(t)$ has uniform bounds, then
this flow exists in the time interval $[0,+\infty)$. 
\end{lemma}

\begin{proof}
Since $\gamma_0$
is an $\ell$-convex Legendre curve, $|\beta(\theta,0)|\le M_0$
holds for a constant $M_0$ only depending on $\gamma_0$.
Equation \eqref{eqn2.9} is uniformly parabolic when
$\lambda(t)$ is uniformly bounded.
By the maximum principle for uniformly parabolic equations, one has $|\beta(\theta,t)|\le M_0$ 
on $\mathbb{S}^1\times [0,T)$. From the standard regularity theory for uniformly parabolic
equations (see Krylov \cite{K1987}), it yields that 
\begin{equation*}
    |\beta^{(i)}|\le M_i
\end{equation*}
holds for $i\ge 1$, where 
$M_i$ is a constant independent of time. Hence, the flow \eqref{jh} exists in the time interval $[0,+\infty)$.    
\end{proof}

\begin{lemma}\label{lem2.4}
Under flow \eqref{jh} when $f=\beta-\lambda(t)$,
if the evolving curve is always $\ell$-convex
and the flow exists in the time interval $[0,+\infty)$, then
\begin{align*}
    \left|\beta-\frac{L}{2\pi}\right|\le \Lambda_1 e^{-3t},
\end{align*}
where $\Lambda_1$ is a constant only depending on 
$\gamma_0$.
\end{lemma}

\begin{proof}
It follows from \eqref{eqn2.9} and integration by parts that
\begin{align*}
&\frac{d}{dt}\int_0^{2\pi}\left(\beta-\frac{L}{2\pi}\right)^2d\theta\\
=&2 \int_0^{2\pi}\left(\beta-\frac{L}{2\pi}\right)\left(\beta_t-\frac{L_t}{2\pi}\right)d\theta\\
=&2\int_0^{2\pi}\left(\beta-\frac{L}{2\pi}\right)\left(\beta_t-\frac{L_t}{2\pi}\right)d\theta\\
=&2\int_0^{2\pi}\left(\beta-\frac{L}{2\pi}\right)\left(\beta_{\theta\theta}+\beta-\lambda(t)-\frac{L_t}{2\pi}\right)d\theta\\
=&-2\int_0^{2\pi}\left(\beta-\frac{L}{2\pi}\right)_{\theta}^2d\theta+2\int_0^{2\pi}\left(\beta-\frac{L}{2\pi}\right)^2d\theta\\
&+2\left(\frac{L}{2\pi}-\lambda(t)-\frac{L_t}{2\pi}\right)\int_0^{2\pi}\left(\beta-\frac{L}{2\pi}\right)d\theta
\end{align*}
Note that 
\begin{align*}
    &\int_0^{2\pi}\left(\beta-\frac{L}{2\pi}\right)d\theta=0,\\
    &\int_0^{2\pi}\left(\beta-\frac{L}{2\pi}\right)\cos\theta d\theta=0,\\
    &\int_0^{2\pi}\left(\beta-\frac{L}{2\pi}\right)\sin \theta d\theta=0,
\end{align*}
one has
\begin{align}\label{gjgc}
    \int_0^{2\pi}\left(\beta-\frac{L}{2\pi}\right)_{\theta}^2d\theta\ge4 \int_0^{2\pi}\left(\beta-\frac{L}{2\pi}\right)^2d\theta.
\end{align}
Together with \eqref{gjgc}, it has
\begin{align*}
 \frac{d}{dt}\int_0^{2\pi}\left(\beta-\frac{L}{2\pi}\right)^2d\theta\le -6  \int_0^{2\pi}\left(\beta-\frac{L}{2\pi}\right)^2d\theta, 
\end{align*}
which deduces that
\begin{align*}
  \int_0^{2\pi}\left(\beta-\frac{L}{2\pi}\right)^2d\theta\le \int_0^{2\pi}\left(\beta_0(\theta)-\frac{L_0}{2\pi}\right)^2d\theta\cdot e^{-6t}.   
\end{align*}
From the above expression and the Sobolev inequality in \cite[p.90]{G-H1986},
the desired estimate is achieved.
\end{proof}

\section{The area-preserving flow}\label{sec4}

When the algebraic area of initial curve $A_0>0$,
taking
$\lambda(t)=\frac{1}{L}\int_{0}^{2\pi}\beta^2 d\theta$ in \eqref{model} which is equivalent to letting $f=\beta-\lambda(t)$ in \eqref{jh}, 
we can deal with the equivalent evolution problem to \eqref{model} as follows
\begin{equation}\label{jh2}
    \begin{cases}
        \frac{\partial \gamma}{\partial t}=(\beta-\lambda(t))'\mu+(\beta-\lambda(t)) \nu, \\
        \frac{\partial \nu}{\partial t}=0.
    \end{cases}
\end{equation}

\begin{lemma}\label{jx-2}
Under flow \eqref{jh2}, the algebraic length
of $\gamma(\theta,t)$ is decreasing and the associated algebraic area is fixed.
\end{lemma}

\begin{proof}
From \eqref{L} and \eqref{ql}, it follows that
\begin{align*}
    \frac{dL}{dt}&=\int_0^{2\pi} \beta d\theta-\frac{2\pi}{L}\int_0^{2\pi}\beta^2 d\theta\\
    &=\frac{1}{L}\left(\left(\int_0^{2\pi} \beta d\theta\right)^2-2\pi \int_0^{2\pi} \beta^2 d\theta\right).
\end{align*}
Combining with the Cauchy-Schwarz inequality,
it yields $\frac{dL}{dt}\le 0$, which implies that
the algebraic length of $\gamma(\theta,t)$ is decreasing.  
By \eqref{eqn2.7}, direct calculation shows that the algebraic area
of $\gamma(\theta,t)$ is fixed.
\end{proof}

\begin{proposition}\label{pro4.1}
Under flow \eqref{jh2},  the nonlocal term $\lambda(t)$ is uniformly bounded.  
\end{proposition}

\begin{proof}
Lemma \ref{jx-2} and \eqref{in-0} tell us
\begin{align*}
   L_0\ge  L(t)\ge 2\sqrt{\pi A(t)}=2\sqrt{\pi A_0}.
\end{align*}

To show the bounds for the term $\lambda(t)$, consider the quantity $F=\int_{0}^{2\pi } \beta ^{2} d\theta$. From integration by parts, \eqref{eqn2.9}, \eqref{L} and \eqref{ql}, it yields
\begin{align*}
\frac{dF}{ dt}&=2\int_{0}^{2\pi } \beta\beta _{t}d\theta  \\
&=2\int_{0}^{2\pi } \beta\left ( \beta _{\theta \theta }+\beta -\lambda \left ( t \right )    \right ) d\theta \\
&=-2\int_{0}^{2\pi }\beta _{\theta }^{2} d\theta  +2\int_{0}^{2\pi }\beta^{2} d\theta-2\lambda \left ( t \right ) L\\
&=-2\int_{0}^{2\pi }\beta _{\theta }^{2} d\theta+2\int_{0}^{2\pi }\beta ^{2} d\theta-2\int_{0}^{2\pi }\beta ^{2} d\theta\le 0.
\end{align*}
This implies that $\int_{0}^{2\pi } \beta ^{2} d\theta\le \int_{0}^{2\pi } \beta ^{2}(\theta,0) d\theta$. 
Together with the inequality \eqref{in-1} 
and the fact that the algebraic area of the
evolving curve is fixed, the desired bounds
for $\lambda(t)$ is achieved.
\end{proof}

\begin{proposition}\label{pro4.2}
If the initial $\ell$-convex Legendre curve $\gamma_0$ is of algebraic area $A_0>0$, then the limiting curve is a circle with radius $\sqrt{\frac{A_0}{\pi}}$ under flow \eqref{jh2}. 
\end{proposition}

\begin{proof}
{\noindent \bf Step 1.} $\beta\rightarrow \sqrt{\frac{A_0}{\pi}}$ as $t\rightarrow \infty$

By \eqref{eqn2.9} and integration by parts, one has
\begin{align*}
\frac{d^{2}F }{dt^{2} }
&=-4\int_{0}^{2\pi }\beta _{\theta }\beta _{t\theta }  d\theta  \\
&=-4\int_{0}^{2\pi } \beta _{\theta } \left ( \beta _{\theta \theta }+\beta -\lambda \left ( t \right )   \right )_{\theta}d\theta  \\
&=4\int_{0}^{2\pi }  \beta _{\theta \theta }^2 d\theta-4\int_{0}^{2\pi } \beta _{\theta}^2 d\theta.
\end{align*}
Together with \eqref{gdds} and the Wirtinger inequality, it has
\begin{align*}
\frac{d^{2}F }{dt^{2} }
&\ge 12 \int_{0}^{2\pi } \beta _{\theta }^2 d\theta=-6\frac{dF }{dt }.
\end{align*}
This deduces that
\begin{align*}
    \frac{dF }{dt }(\theta,t)\ge \frac{dF }{dt }(\theta,0)\ e^{-6t},
\end{align*}
which combining with $\frac{dF}{dt}\le 0$ 
yields
$\frac{dF }{dt }\rightarrow0$ as time $t\rightarrow\infty$, Thus, $\beta_\theta\rightarrow 0$ and $\beta$ tends to a constant $r$ as time $t$ goes to infinite.

Combining with \eqref{L}, \eqref{A} and the fact that $\beta\rightarrow r$
as $t\rightarrow \infty$, it yields $r=\sqrt{\frac{A_0}{\pi}}$.

{\noindent \bf Step 2.}  The limiting curve is a circle

Let $p(\theta,t)$ be the support function of 
the evolving curve $\gamma(\theta,t)$.
Set $$a_1(t)=\frac{1}{\pi}\int_0^{2\pi}p(\theta,t)\cos\theta d\theta\quad \text{and}\quad b_1(t)=\frac{1}{\pi}\int_0^{2\pi}p(\theta,t)\sin\theta d\theta.$$
By \eqref{eqn2.6} and integration by parts, one has
\begin{align*}
   \frac{da_1(t)}{dt}&= \frac{1}{\pi}\int_0^{2\pi}p_t\cos\theta d\theta\\
   &=\frac{1}{\pi}\int_0^{2\pi}(p_{\theta \theta}+p-\lambda(t))\cos\theta d\theta\\
   &=-\frac{1}{\pi}\int_0^{2\pi}p\cos\theta d\theta+\frac{1}{\pi}\int_0^{2\pi}p\cos\theta d\theta=0.
\end{align*}
In exactly the same way, $\frac{db_1(t)}{dt}=0$. These results
deduce that $a_1(t)$ and $b_1(t)$ are constants 
independent of time, that is, point $(a_1(t),b_1(t))$ is fixed under flow \eqref{jh2}, denoted by $(a_1,b_1)$. 

Suppose that $\widetilde{p}(\theta,t)=p(\theta,t)-a_{1}\cos \theta -b_{1}  \sin \theta$, then 
\begin{align}\label{tj}
\int_{0}^{2\pi } \widetilde{p}(\theta,t)\cos\theta d\theta=\int_{0}^{2\pi } \widetilde{p}(\theta,t)\sin\theta d\theta=0.  
\end{align}
Again by \eqref{eqn2.6}, one has
\begin{align}\label{fcbs}
\widetilde{p}_{t}={p}_{t}&={p}_{\theta \theta }+p-\lambda \left ( t \right )
=\widetilde{p}_{\theta \theta }+ \widetilde{p}-\lambda(t).
\end{align}

Consider the quantity $G=\int_{0}^{2\pi } \widetilde{p} ^{2} d\theta$. It follows from \eqref{fcbs}, integration
by parts and the fact $\beta=\widetilde{\beta}$ that
\begin{align*}
\frac{dG}{dt}&=2\int_{0}^{2\pi } \widetilde{p} \widetilde{p} _{t}d\theta  \\
&=2\int_{0}^{2\pi } \widetilde{p}\left ( \widetilde{p} _{\theta \theta }+\widetilde{p} - \lambda \left ( t \right )  \right ) d\theta \\
&=-2\int_{0}^{2\pi }\widetilde{p} _{\theta }^{2} d\theta  +2\int_{0}^{2\pi }\widetilde{p}^{2} d\theta-2\int_{0}^{2\pi } \beta ^{2} d\theta \\
&=-2\int_{0}^{2\pi }\widetilde{p} _{\theta }^{2} d\theta  +2\int_{0}^{2\pi }\widetilde{p}^{2} d\theta-2\int_{0}^{2\pi } \left ( \widetilde{p}+\widetilde{p}_{\theta \theta } \right )   ^{2} d\theta \\
&=-2\int_{0}^{2\pi }\widetilde{p} _{\theta }^{2} d\theta  +2\int_{0}^{2\pi }\widetilde{p}^{2} d\theta-2\int_{0}^{2\pi }\widetilde{p}^{2} d\theta-2\int_{0}^{2\pi }\widetilde{p} _{\theta\theta }^{2} d\theta+4\int_{0}^{2\pi }\widetilde{p} _{\theta}^2 d\theta \\
&=-2\int_{0}^{2\pi }\widetilde{p} _{\theta\theta }^{2} d\theta+2\int_{0}^{2\pi }\widetilde{p}_{\theta }^{2} d\theta.  
\end{align*}
Together with the Wirtinger inequality, this yields $\frac{dG}{dt}\le 0$.

Let $\widetilde{p}^{(i)}$ be the $i$-th derivative of $\widetilde{p}$.
Next, we show that $\widetilde{p}^{(i)}$ has uniform bounds. 
Indeed, equation \eqref{fcbs} is uniformly parabolic
since $\lambda(t)$ is uniformly bounded.
By the maximum principle for uniformly parabolic equations, $|\widetilde{p}(\theta,t)|\le C_0$ holds for a
constant $C_0$ only depending on $\gamma_0$. From the standard regularity theory for uniformly parabolic
equations (see Krylov \cite{K1987}), it yields
$|\widetilde{p}^{(i)}|\le C_i$
holds for $i\ge 1$, where 
$C_i$ is a constant independent of time.

One can compute 
\begin{align*}
\frac{d^2G}{dt^2}=4\int_0^{2\pi}\widetilde{p}_{\theta\theta\theta}^2d\theta  -8\int_0^{2\pi}\widetilde{p}_{\theta\theta}^2d\theta+4\int_0^{2\pi}\widetilde{p}_{\theta}^2d\theta. 
\end{align*}
Since all the terms of $\frac{d^2G}{dt^2}$ have bounds independent of time,
$$\left|\frac{d^2G}{dt^2}\right|\le M$$
holds for a constant $M$ independent of time $t$.

Noticing that $\frac{dG}{dt}$ is nonpositive, the quantity
$$\int_0^{\infty} \frac{dG}{dt} dt=G(\infty)-G(0)$$
has a uniformly lower bound. 
Then it follows that 
$$\lim_{t\rightarrow \infty}\frac{dG}{dt}=0.$$
Hence, again by \eqref{L} and \eqref{A}, one
can get
$\lim\limits_{t\rightarrow \infty}\widetilde{p}=\sqrt{\frac{A_0}{\pi}}$,
that is,  the limiting curve is a circle
centered at $(a_1,b_1)$ with radius $\sqrt{\frac{A_0}{\pi}}$.
\end{proof}

\begin{remark}
The asymptotic behavior of $\beta$ for
the area-preserving flow \eqref{jh2}
can be also achieved by the quantity $\frac{dL}{dt}$. In fact, the boundedness for high order derivatives of $\beta$ can deduce $\frac{d^2L}{dt^2}$ is uniformly bounded. Together
with the fact $\frac{dL}{dt}\le0$
and $$\int_0^{\infty} \frac{dL}{dt} dt=L(\infty)-L(0)$$
has a uniformly lower bound, it yields
$\lim\limits_{t\rightarrow \infty}\frac{dL}{dt}=0$.
This implies that $\beta$ tends to a constant as time $t$ goes to infinity due to the Cauchy-Schwarz inequality.
\end{remark}

\begin{lemma}\label{lem4.3}
Under flow \eqref{jh2}, $|\beta^{(i)}|$ exponentially decays for $i\ge 1$.    
\end{lemma}

\begin{proof}
It follows from \eqref{eqn2.9} and integration by parts that 
\begin{align*}
\frac{d }{dt} \int_{0}^{2\pi } \left ( \beta ^{ \left ( i \right )}   \right )  ^{2} d\theta 
=&2\int_{0}^{2\pi }\beta ^{\left ( i \right ) } \beta _{t}^{\left ( i \right ) }   d\theta \\
=&2\int_{0}^{2\pi } \beta ^{\left ( i \right ) } \left ( \beta _{\theta \theta }+\beta -\lambda \left ( t \right )   \right )^{ \left ( i \right ) } d\theta  \\
=&2\int_{0}^{2\pi } \beta  ^{\left ( i \right ) }\beta  ^{\left ( i+2 \right ) } d\theta +2\int_{0}^{2\pi } \left(\beta  ^{(i)} \right) ^{2} d\theta\\
=&-2\int_{0}^{2\pi } \left(\beta  ^{(i+1)} \right) ^{2} d\theta +2\int_{0}^{2\pi }  \left(\beta  ^{(i)} \right) ^{2} d\theta.
\end{align*}
Together with the inequality
\begin{align*}
\int_{0}^{2\pi } \left(\beta  ^{(i+1)} \right) ^{2} d\theta\ge 4\int_{0}^{2\pi } \left(\beta  ^{(i)} \right) ^{2} d\theta    \qquad \text{for $i\ge 1$ }
\end{align*}
derived from the Wirtinger inequality and \eqref{gdds}, it yields
\begin{align*}
\frac{d }{dt} \int_{0}^{2\pi } \left ( \beta ^{ \left ( i \right )}   \right )  ^{2} d\theta\le -6\int_{0}^{2\pi } \left(\beta  ^{(i)} \right) ^{2} d\theta.
\end{align*}
This implies that
\begin{align*}
    \int_{0}^{2\pi } \left(\beta  ^{(i)}(\theta,t) \right) ^{2} d\theta\le \int_{0}^{2\pi } \left(\beta  ^{(i)}(\theta,0) \right) ^{2} d\theta\cdot e^{-6t}. 
\end{align*}
Thus, by the Sobolev inequality appearing in \cite[p.90]{G-H1986}, $|\beta^{(i)}|$ exponentially decays for $i\ge 1$. 
\end{proof}

In order to show that the evolving curve $\gamma(\cdot,t)$ cannot escape to infinity,
we need to prove the next necessary lemma. 

\begin{lemma}\label{lem-sdx}
Under flow \eqref{jh2}, one can get
\begin{align*}
    \left|\frac{L^2}{2\pi}-\int_0^{2\pi}\beta^2 d\theta\right|\le \Lambda_2 e^{-6t},
\end{align*}
where $\Lambda_2$ is a constant only depending on $\gamma_0$.
\end{lemma}

\begin{proof}
Consider the quantity $Q=\frac{L^2}{2\pi}-\int_0^{2\pi}\beta^2 d\theta$. We first
claim that $Q\le 0$. Taking $F(x)=x^2$ 
in the Green-Osher inequality for $\ell$-convex
Legendre curves (see \cite[(1.10)]{L-W2023}),
it has
\begin{align*}
    \int_0^{2\pi}\beta^2d\theta \ge \frac{L^2-2\pi A}{\pi}.
\end{align*}
From the isoperimetric inequality
for $\ell$-convex Legendre curves, it yields 
\begin{align*}
    Q\le \frac{L^2}{2\pi}-\frac{L^2-2\pi A}{\pi}\le -\frac{L^2-4\pi A}{2\pi}\le 0.
\end{align*}
By \eqref{eqn2.9}, \eqref{L}, \eqref{ql} and \eqref{eqn2.7}, one can compute
\begin{align*}
    \frac{dQ}{dt}&=\frac{L}{\pi}L_t-2\int_0^{2\pi}\beta\beta_t d\theta\\
    &=\frac{L}{\pi}L_t-2\int_0^{2\pi}\beta(\beta_{\theta\theta}+\beta-\lambda(t)) d\theta\\
    &=\frac{L}{\pi}\int_0^{2\pi}(\beta-\lambda(t))d\theta+2\int_0^{2\pi}\beta_{\theta}^2 d\theta-2\int_0^{2\pi}\beta^2 d\theta+2\lambda(t)\int_0^{2\pi}\beta d\theta\\
    &=\frac{L}{\pi}(L-2\pi \lambda(t))+2\int_0^{2\pi}\beta_{\theta}^2 d\theta-2\int_0^{2\pi}\beta^2 d\theta+2L\lambda(t)\\
    &=2\left(\frac{L^2}{2\pi}-\int_0^{2\pi}\beta^2 d\theta\right)+2\int_0^{2\pi}\beta_{\theta}^2 d\theta\\
    &=2Q+2\int_0^{2\pi}\beta_{\theta}^2 d\theta.
\end{align*}
Since inequality \eqref{gjgc} holds,
this can be written as
\begin{align*}
   \frac{dQ}{dt}&\ge 2Q+8 \int_0^{2\pi}\left(\beta-\frac{L}{2\pi}\right)^2 d\theta\\
   &=2Q+8\int_0^{2\pi}\beta^2 d\theta-\frac{8L}{\pi}\int_0^{2\pi}\beta d\theta+\frac{4L^2}{\pi}\\
   &=2Q+8\int_0^{2\pi}\beta^2 d\theta-\frac{4L^2}{\pi}\\
   &=2Q-8Q=-6Q,
\end{align*}
which together with $Q\le 0$ deduces that
\begin{align*}
    \frac{dQ^2}{dt}=2Q \frac{dQ}{dt}\le-12Q^2. 
\end{align*}
This gives
$|Q|\le |Q(0)|e^{-6t}$,
the desired result is concluded.
\end{proof}

\begin{proposition}\label{jh2-sl}
Under flow \eqref{jh2}, the evolving curve converges to a limiting
circle as time goes to infinity.    
\end{proposition}

\begin{proof}
From the isoperimetric inequality for $\ell$-convex Legendre curves, it yields 
\begin{align*}
\left|\beta-\frac{1}{L}\int_0^{2\pi}\beta^2 d\theta\right|=&\left|\beta-\frac{L}{2\pi}+\frac{L}{2\pi}-\frac{1}{L}\int_0^{2\pi}\beta^2 d\theta\right|\\
\le& \left|\beta-\frac{L}{2\pi}\right|+\left|\frac{L}{2\pi}-\frac{1}{L}\int_0^{2\pi}\beta^2 d\theta\right|\\
=&\left|\beta-\frac{L}{2\pi}\right|+\frac{1}{L}\left|\frac{L^2}{2\pi}-\int_0^{2\pi}\beta^2 d\theta\right|\\
\le &\left|\beta-\frac{L}{2\pi}\right|+\frac{1}{
\sqrt{4\pi A_0}}\left|\frac{L^2}{2\pi}-\int_0^{2\pi}\beta^2 d\theta\right|.
\end{align*}
Using Lemma \ref{lem2.4}, Lemma \ref{lem4.3} and Lemma \ref{lem-sdx}, there 
exists positive constants $\Lambda$ and $\alpha$
such that
\begin{align}\label{qxgj}
    \left|\frac{\partial \gamma}{\partial t}\right|\le \Lambda e^{-\alpha t}.
\end{align}
This implies that
the evolving curve $\gamma(\cdot,t)$ cannot escape to infinity.  Due to the
Blaschke selection theorem, there exists a subsequence $\{t_i\}$ such that $\gamma(\cdot,t_i)$ converges
to a circle with radius $\sqrt{\frac{A_0}{\pi}} $ denoted by $\gamma_{\infty}$ as $t_i\rightarrow\infty$.  
Integrating the first equation in \eqref{jh2}, we get
\begin{align*}
    \gamma_\infty=\int_0^{\infty}\frac{\partial \gamma}{\partial t}dt+\gamma_0&\le \int_0^{\infty}\left|\frac{\partial \gamma}{\partial t}\right|dt+\gamma_0\\
    &\le \frac{\Lambda}{\alpha}+\gamma_0.
\end{align*}
This tells us the limit curve $\gamma_{\infty}$ will not escape to infinity as time $t\rightarrow\infty$.

Suppose that $t_i>t$, then by \eqref{qxgj}, it has
\begin{align*}
    |\gamma(\cdot,t)-\gamma(\cdot,t_i)|\le\int_t^{t_i}\left|\frac{\partial \gamma}{\partial \tau}\right|d\tau\le \frac{\Lambda}{\alpha}(e^{-\alpha t}-e^{-\alpha t_i}).
\end{align*}
This deduces that
\begin{align*}
    |\gamma(\cdot,t)-\gamma_\infty|\le \frac{\Lambda}{\alpha}e^{-\alpha t}
\end{align*}
as time $t_i\rightarrow \infty$, which shows that the
evolving curve cannot oscillate indefinitely and converges to a limit circle.
\end{proof}


{\bf \noindent Proof of Theorem \ref{thm1.1}}\quad
Since adding the tangential vector field for flow \eqref{model} when $\lambda(t)=\frac{1}{L}\int_0^{2\pi}\beta^2 d\theta$ does not affect the geometric shape of the evolving curve, flow \eqref{model} can be equivalently reduced to \eqref{jh2}. Proposition \ref{pro4.1} shows that the term $\lambda(t)$ is uniformly bounded. Lemma \ref{lem2.2} and Lemma \ref{lem2.3} implies that an $\ell$-convex Legendre curve evolving according to \eqref{jh2} remains so and the long time
existence for this flow is obtained.
Proposition \ref{jh2-sl} deduces that the evolving curve $\gamma(\cdot,t)$ cannot escape to infinity. Meanwhile, Proposition \ref{pro4.2}
implies that the evolving curve converges to a circle of radius $\sqrt{\frac{A_0}{\pi}}$. Finally, Lemma \ref{lem4.3} 
says that the flow \eqref{jh2} converges
in the $C^{\infty}$ sense as time $t$ goes to
infinity.\qed

\section{The length-preserving flow}\label{sec3}

Choosing
$\lambda(t)=\frac{L}{2\pi}$ in \eqref{model} that is equivalent to letting $f=\beta-\frac{L}{2\pi}$ in \eqref{jh}, the
flow \eqref{model} becomes the following equivalent evolution problem
\begin{equation}\label{jh1}
    \begin{cases}
        \frac{\partial \gamma}{\partial t}=(\beta-\frac{L}{2\pi})'\mu+(\beta-\frac{L}{2\pi}) \nu, \\
        \frac{\partial \nu}{\partial t}=0.
    \end{cases}
\end{equation}

\begin{lemma}\label{jx-1}
Under flow \eqref{jh1}, the algebraic length
of $\gamma(\theta,t)$ keeps fixed and the associated algebraic area is increasing.
\end{lemma}

\begin{proof}
The algebraic length
of $\gamma(\theta,t)$ is fixed derived from \eqref{eqn2.7}. 
By \eqref{eqn2.8}, \eqref{L} and \eqref{ql}, it yields that
\begin{align}\label{At-yy}
    \frac{dA}{dt}=\int_0^{2\pi} \beta^2 d\theta-\frac{L^2}{2\pi},
\end{align}
Together with the Cauchy-Schwarz inequality,
it yields $\frac{dA}{dt}\ge 0$, which deduces that
the algebraic area of $\gamma(\theta,t)$ is increasing.
\end{proof}

\begin{proposition}\label{pro3.2}
Suppose that the initial $\ell$-convex Legendre curve is of algebraic length $L_0$. Under flow \eqref{jh1}, 
the evolving curve converges to a circle with radius $\frac{L_0}{2\pi}$. Specially, if $L_0=0$, then the limiting circle becomes a point.
\end{proposition}

\begin{proof}
{\bf Step 1.}    $\beta\rightarrow \frac{L_0}{2\pi}$ as $t\rightarrow \infty$

Since flow \eqref{jh1} is length-preserving, by Lemma \ref{lem2.3} and Lemma \ref{lem2.4}, $\beta\rightarrow \frac{L_0}{2\pi}$ as $t\rightarrow \infty$.

{\noindent \bf Step 2.}  The limiting curve is a circle

Let $p(\theta,t)$ be the support function of 
the evolving curve $\gamma(\theta,t)$.
Set $$a_1(t)=\frac{1}{\pi}\int_0^{2\pi}p(\theta,t)\cos\theta d\theta\quad \text{and}\quad b_1(t)=\frac{1}{\pi}\int_0^{2\pi}p(\theta,t)\sin\theta d\theta.$$
A simple computation as in Proposition \ref{pro4.2} tells us $a_1(t)$ and $b_1(t)$ are constants 
independent of time, denoted by $(a_1,b_1)$.

Assume that $\overline{p}(\theta,t)=p(\theta,t)-a_{1}\cos \theta -b_{1}  \sin \theta$, then $\overline{\beta}=\beta$ and
\begin{align*}
  \overline{p}_t= \overline{p}_{\theta\theta}+\overline{p}-\frac{L}{2\pi}.
\end{align*}
Noticing that 
\begin{align*}
    \int_0^{2\pi}\left(\overline{p}-\frac{L}{2\pi}\right)d\theta=
    \int_0^{2\pi}\left(\overline{p}-\frac{L}{2\pi}\right)\cos\theta d\theta=
    \int_0^{2\pi}\left(\overline{p}-\frac{L}{2\pi}\right)\sin \theta d\theta=0,
\end{align*}
one has
\begin{align*}
    \int_0^{2\pi}\left(\overline{p}-\frac{L}{2\pi}\right)_{\theta}^2d\theta\ge4 \int_0^{2\pi}\left(\overline{p}-\frac{L}{2\pi}\right)^2d\theta.
\end{align*}
Then,
\begin{align*}
\frac{d}{ dt}\int_{0}^{2\pi } \left(\overline{p}-\frac{L}{2\pi}\right)^{2} d\theta
&=2\int_{0}^{2\pi } \left(\overline{p}-\frac{L}{2\pi}\right)\overline{p}_{t}d\theta  \\
&=2\int_{0}^{2\pi } \left(\overline{p}-\frac{L}{2\pi}\right)\left(\overline{p}_{\theta\theta}+\overline{p}-\frac{L}{2\pi}\right) d\theta \\
&=2\int_{0}^{2\pi }\left(\overline{p}-\frac{L}{2\pi}\right)^2 d\theta  -2\int_{0}^{2\pi }\left(\overline{p}-\frac{L}{2\pi}\right)_{\theta}^2 d\theta\\
&\le -6\int_{0}^{2\pi }\left(\overline{p}-\frac{L}{2\pi}\right)^2 d\theta,
\end{align*}
which implies that $\overline{p}\rightarrow\frac{L_0}{2\pi}$
as $t\rightarrow \infty$, that is, the limiting curve is a circle of radius $\frac{L_0}{2\pi}$ centered at $(a_1,b_1)$.

Specially, if $L_0=0$, then the limit circle turns into a point $(a_1,b_1)$ that is circle of radius zero.
\end{proof}

By a similar proof as in Lemma \ref{lem4.3}, one can get 

\begin{lemma}\label{lem3.3}
Under flow \eqref{jh1}, $|\beta^{(i)}|$ exponentially decays for $i\ge 1$.    
\end{lemma}

Combining Lemma \ref{lem2.3} and Lemma \ref{lem3.3}, we have

\begin{proposition}\label{jh2-s2}
Under flow \eqref{jh1}, the evolving curve converges to a limiting
circle as time goes to infinity.    
\end{proposition}

\begin{proof}
Since the proof is alomst the same as Proposition \ref{jh2-sl}, we omit the detailed reasoning procession.     
\end{proof}

{\bf \noindent Proof of Theorem \ref{thm1.2}}\quad
Since adding the tangential vector field for flow \eqref{model} when $\lambda(t)=\frac{2\pi}{L}$ does not affect the geometric shape of the evolving curve, flow \eqref{model} can be equivalently reduced to \eqref{jh1}. 
Since flow \eqref{jh1} is length-preserving,
Lemma \ref{lem2.2} and Lemma \ref{lem2.3} implies that an $\ell$-convex Legendre curve evolving according to \eqref{jh1} remains so and the long time
existence for this flow is achieved.
Proposition \ref{jh2-s2} deduces that the evolving curve $\gamma(\cdot,t)$ cannot escape to infinity. Then, Proposition \ref{pro3.2}
implies that the evolving curve converges to a circle of radius $\frac{L_0}{2\pi}$ and the limit circle becomes a point when $L_0=0$. At last, Lemma \ref{lem3.3} 
says that the flow \eqref{jh1} converges
in the $C^{\infty}$ sense as time $t$ goes to
infinity.
\qed

\section{Some applications}\label{sec5}

In this section, we obtain some geometric 
inequalities for
$\ell$-convex Legendre curves via the length-preserving inverse curvature flow \eqref{jh1}.

First, we show the isoperimetric inequality
for $\ell$-convex curves.

\begin{theorem}\label{thm5.0}
Let  $\gamma$ be an $\ell$-convex curve with algebraic length $L$ and algebraic area $A$,
then   
\begin{equation}\label{eqn5.0}
 L^2- 4\pi A \ge 0
\end{equation}
with equality if and only if $\gamma$
is a circle. 
\end{theorem}

\begin{proof}
Under flow \eqref{jh1}, since it is 
length-preserving, from Lemma \ref{jx-1} and the Cauchy-Schwarz inequality,
the quantity $U=L^2-4\pi A$ is strict decreasing
unless the evolving curve is a circle. 
Thus, $U(t)\ge U(\infty)=0$ and equality holds
if and only if the evolving curve is a circle.
This deduces the desired result.
\end{proof}

Next, we show the generalizations
for inequalities \eqref{in-1}, \eqref{in-2}
and \eqref{in-3}.

\begin{proposition}\label{pro5.1}
Let  $\gamma$ be an $\ell$-convex curve of algebraic length $L$ and algebraic area $A$,
then
\begin{equation}\label{eqn5.1}
    \int_0^{2\pi }\beta^2(\theta)d\theta\ge 2A+\tau\left(\frac{L^2}{4\pi}- A\right)
\end{equation}
holds for any $\tau\le 8$, 
and the equality holds if $\gamma$ is a circle.
Moreover, for any $\tau<8$, if the equality in \eqref{eqn5.1} holds, then $\gamma$ is a circle,
and for $\tau=8$, if the equality holds,
then the support function of $\gamma$ is of the form $p(\theta)=a_0+a_1\cos\theta+b_1\sin\theta+a_2\cos 2\theta+b_2\sin2\theta$.
\end{proposition}

\begin{proof}
Consider the quantity $W=\int_0^{2\pi}\beta^2 d\theta-2A-\tau\left(\frac{L^2}{4\pi}- A\right)$. Compute that
\begin{align*}
    \frac{dW}{dt}=2\int_0^{2\pi}\beta\beta_t d\theta+(\tau-2)A_t.
\end{align*}
From \eqref{eqn2.12}, integration by parts and 
\eqref{At-yy}, the above expression
can be rewritten as
\begin{align*}
    \frac{dW}{dt}=-2\int_0^{2\pi}\beta_{\theta}^2 d\theta+\tau A_t.
\end{align*}
Together with \eqref{gjgc}, \eqref{L} and \eqref{At-yy}, this yields
\begin{align*}
    \frac{dW}{dt}\le(\tau-8)A_t=(\tau-8)\left(\int_0^{2\pi}\beta^2 d\theta-\frac{L^2}{2\pi}\right).
\end{align*}

When $\tau<8$, due to the Cauchy-Schwarz inequality, the quantity $W$ is strict decreasing unless the evolving curve is a circle.  Thus,
$W(t)\ge W(\infty)=0$, which deduces inequality \eqref{eqn5.1} and with equality holds if and only if the evolving curve is a circle.

Note that the inequality \eqref{gjgc} is strict unless the support function of evolving curve
is of form $p(\theta,t)=\frac{L}{2\pi}+a_1(t)\cos\theta+b_1(t)\sin\theta+a_2(t)\cos 2\theta+b_2(t)\sin2\theta$.
This implies that the quantity $W$ is strict decreasing in the case $\tau=8$. Hence,
$W(t)\ge W(\infty)=0$, that is, \eqref{eqn5.1}
holds and with equality if and only if the support function of evolving curve
is of form $p(\theta,t)=\frac{L}{2\pi}+a_1(t)\cos\theta+b_1(t)\sin\theta+a_2(t)\cos 2\theta+b_2(t)\sin2\theta$.
\end{proof}

By \eqref{eqn5.1} and \eqref{L1}, we have
\begin{corollary}\label{cor5.2}
Let  $\gamma$ be an $\ell$-convex curve with algebraic length $L=0$ and algebraic area $A$,
then
\begin{equation}\label{eqn5.2cr}
    \int_0^{2\pi }\beta^2(\theta)d\theta +\tau A\ge 0
\end{equation}
holds for any $\tau\le 6$, 
and the equality holds if $\gamma$ is a point.
Moreover, for any $\tau<6$, if the equality in \eqref{eqn5.2cr} holds, then $\gamma$ is a point,
and for $\tau=6$, if the equality  holds,
then the support function of $\gamma$ is of the form $p(\theta)=a_1\cos\theta+b_1\sin\theta+a_2\cos 2\theta+b_2\sin2\theta$.
\end{corollary}

\begin{theorem}\label{thm-j1}
If $\gamma$ is an $\ell$-convex curve, then 
\begin{equation}\label{in-j1}
    \int_0^{2\pi} \beta_\theta^2 d\theta\ge 
        \xi\left(\frac{L^2}{4\pi}-A\right)
\end{equation}
holds for $\xi\le 24$,
and the equality holds if $\gamma$ is a circle.
Moreover, for any $\xi<24$, if the equality in \eqref{in-j1} holds, then $\gamma$ is a circle,
and for $\xi=24$, 
if the equality in \eqref{in-j1} holds
if and only if $\gamma$ is of support function
$$p(\theta)=a_0+a_1\cos\theta+b_1\sin\theta+a_2\cos2\theta+b_2\sin2\theta.$$
\end{theorem}

\begin{proof}
Consider the quantity $V=\int_0^{2\pi}\beta_\theta^2 d\theta-\xi\left(\frac{L^2}{4\pi}- A\right)$.
It follows from \eqref{eqn2.12}, integration by parts and 
\eqref{At-yy} that
\begin{align*}
    \frac{dV}{dt}=-2\int_0^{2\pi}\beta_{\theta\theta}^2 d\theta+\int_0^{2\pi}\beta_{\theta}^2 d\theta+\xi A_t.
\end{align*}
Due to \eqref{gdds},
\begin{align*}
    &\int_0^{2\pi}\beta_\theta d\theta=0,\\
    &\int_0^{2\pi}\beta_\theta \cos\theta d\theta
    =\int_0^{2\pi}\beta_\theta \sin \theta d\theta
    =0,
\end{align*}
we have
\begin{align}\label{gjgc-1}
    \int_0^{2\pi}\beta_{\theta\theta}^2 d\theta\ge4 \int_0^{2\pi}\beta_{\theta}^2 d\theta.
\end{align}
Combine \eqref{gjgc}, \eqref{L} and \eqref{At-yy}, it yields
\begin{align*}
    \frac{dV}{dt}\le(\xi-24)A_t=(\xi-24)\left(\int_0^{2\pi}\beta^2 d\theta-\frac{L^2}{2\pi}\right).
\end{align*}

For the case $\xi<24$, from the Cauchy-Schwarz inequality, the quantity $V$ is strict decreasing unless the evolving curve is a circle.  Hence,
$V(t)\ge V(\infty)=0$, which deduces \eqref{in-j1}
holds and with equality if and only if the evolving curve is a circle.

Since the inequalities \eqref{gjgc} and \eqref{gjgc-1} are strict unless the support function of evolving curve
is of form $p(\theta,t)=\frac{L}{2\pi}+a_1(t)\cos\theta+b_1(t)\sin\theta+a_2(t)\cos 2\theta+b_2(t)\sin2\theta$,
the quantity $V$ is strict decreasing for the case $\xi=24$. This leads to
$V(t)\ge V(\infty)=0$, and thus, \eqref{in-j1}
holds and with equality if and only if the support function of evolving curve
is of form $p(\theta,t)=\frac{L}{2\pi}+a_1(t)\cos\theta+b_1(t)\sin\theta+a_2(t)\cos 2\theta+b_2(t)\sin2\theta$.    
\end{proof}

As a natural corollary of inequality \eqref{in-j1},
we have
\begin{corollary}\label{cor5.3}
Let  $\gamma$ be an $\ell$-convex curve with algebraic length $L=0$ and algebraic area $A$,
then
\begin{equation}\label{eqn5.3cr}
    \int_0^{2\pi }\beta_\theta^2d\theta +\xi A\ge 0
\end{equation}
holds for any $\xi\le 24$, 
and the equality holds if $\gamma$ is a point.
Moreover, for any $\xi<24$, if the equality in \eqref{eqn5.3cr} holds, then $\gamma$ is a point,
and for $\xi=24$, if the equality  holds,
then the support function of $\gamma$ is of the form $p(\theta)=a_1\cos\theta+b_1\sin\theta+a_2\cos 2\theta+b_2\sin2\theta$.
\end{corollary}



{\bf\noindent Data availability} \quad Data sharing not applicable to this article as no datasets were generated or analysed during
the current study.


{\bf\noindent Conflict of interest} \quad On behalf of all authors, the corresponding author states that there is no conflict of interest.


\begin{thebibliography}{99}


\bibitem{A-C-G-L2020}
B. Andrews, B. Chow, C. Guenther and M. Langford, \emph{Extrinsic Geometric Flows}, AMS, Providence, RI, 2020.


\bibitem{A1988}
S.B. Angenent, 
The zero set of a solution of a parabolic equation, \emph{J. Reine Angrew. Math.} 390 (1988), 79--96.

\bibitem{A1990}
S.B. Angenent,
Parabolic equations for curves on surfaces. I: Curves with \(p\)-integrable curvature. 
\emph{Ann. Math. (2)} 132 (1990),  451--483.



\bibitem{A1991}
S.B. Angenent,
Parabolic equations for curves on surfaces. II: Intersections, blow-up and generalized solutions.
\emph{Ann. Math. (2)} 133 (1991),  171--215.


\bibitem{B-G-L2020}
S. Brendle, P. Guan and J. Li,
An inverse curvature type hypersurface flow in space forms. Preprint (2020).


\bibitem{B-H-W2016}
S. Brendle, P. Hung and M. Wang,
A Minkowski inequality for hypersurfaces in the anti-de Sitter-Schwarzschild manifold.
\emph{Comm. Pure Appl. Math.} \textbf{69} (2016), 124--144.


\bibitem{C-Z2001}
K.S. Chou and X.P. Zhu,
\emph{The Curve Shortening Problem}.
Chapman \& Hall/ CRC, 2001.



\bibitem{D2021}
F. Dittberner, 
Curve flows with a global forcing term.
\emph{J. Geom. Anal.} \textbf{31} (2021), 8414--8459.


\bibitem{F-T2013}
T. Fukunaga and M. Takahashi,  
Existence and uniqueness for Legendre curves. \emph{J. Geom.} \textbf{104} (2013), 297--307. 

\bibitem{F-T2016}
T. Fukunaga and M. Takahashi, 
On convexity of simple closed frontals. 
\emph{Kodai Math. J.} \textbf{39} (2016), 389--398. 




\bibitem{G1986}
M. Gage,
On an area-preserving evolution equation for plane curves. Nonlinear problems in geometry, \emph{Contemp. Math.} \textbf{51} (1986), 51--62. 


\bibitem{G-H1986}
M.E. Gage and R.S. Hamilton,
The heat equation shrinking convex plane curves,
\emph{J. Differ. Geom.} \textbf{23} (1986), 69--96.



\bibitem{G-P-T2020}
L.Y. Gao, S.L. Pan and D.-H. Tsai,
On a length-preserving inverse curvature flow of convex closed plane curves.
\emph{J. Differ. Equations} \textbf{269} (2020), 5802--5831.


\bibitem{G-P-T2021}
L.Y. Gao, S.L. Pan and D.-H. Tsai,
On an area-preserving inverse curvature flow of convex closed plane curves.
\emph{J. Funct. Anal.} \textbf{280} (2021), 108931, 32p.

\bibitem{G-P2023}
L.Y. Gao and S.L. Pan, 
Star-shaped centrosymmetric curves under Gage’s area-preserving flow. 
\emph{J. Geom. Anal.} \textbf{33} (2023), No.348, 25p.


\bibitem{G1990}
C. Gerhardt,
Flow of nonconvex hypersurfaces into spheres.
\emph{J. Differ. Geom.} \textbf{32} (1990),  299--314.


\bibitem{G-G-K-O2022}
M.H. Giga, Y. Giga, R. Kuroda and Y. Ochiai, Crystalline flow starting from a general polygon. 
\emph{Discrete Contin. Dyn. Syst.}, \textbf{42} (2022), 2027--2051.

\bibitem{G-K1985}
Y. Giga and R.V. Kohn, 
Asymptotically self-similar blow-up of semilinear heat equations. 
\emph{Commun. Pure Appl. Math.} \textbf{38} (1985), 297--319.

\bibitem{G1987}
M.A. Grayson, 
The heat equation shrinks embedded plane curves to round points.
\emph{J. Differ. Geom.} \textbf{26} (1987), 285--314.

\bibitem{G-O1999}
M. Green and S. Osher, 
Steiner polynomials, Wulff flows, and some new isoperimetric inequalities for convex plane curves, 
\emph{Asian J. Math.} \textbf{3} (1999), 659--676.


\bibitem{G-L2015}
P. Guan and J. Li,
A mean curvature type flow in space forms.
\emph{Int. Math. Res. Not.} \textbf{2015} (2015), 4716--4740.


\bibitem{G-L2009}
P. Guan and J. Li,
The quermassintegral inequalities for $k$-convex starshaped domains.
\emph{Adv. Math.},  \textbf{221} (2009), 1725-1732.


\bibitem{G-L2018}
P. Guan and J. Li,
A fully-nonlinear flow and quermassintegral inequalities.
\emph{Sci. Sin., Math.} \textbf{48} (2018), 147--156.


\bibitem{G-L2021}
P. Guan and J. Li,
Isoperimetric type inequalities and hypersurface flows.
\emph{J. Math. Study} \textbf{54} (2021),  56--80.







\bibitem{H-L2022}
Y. Hu and H. Li,
Geometric inequalities for static convex domains in hyperbolic space.
\emph{Trans. Amer. Math. Soc.} \textbf{375} (2022), 5587--5615.




\bibitem{H-L-W2022}
Y. Hu, H. Li and Y. Wei,
Locally constrained curvature flows and geometric
inequalities in hyperbolic space.
\emph{Math. Ann.} \textbf{382} (2022), 1425--1474.


\bibitem{H-I2001}
G. Huisken and T. IImanen,
The inverse mean curvature flow and the Riemannian Penrose inequality.
\emph{J. Differ. Geom.} \textbf{59} (2001), 353--437.

\bibitem{J-P2008}
L.S. Jiang and S.L. Pan,
On a non-local curve evolution problem in the plane. 
\emph{Commun. Anal. Geom.} \textbf{16}(2008), 1--26.


\bibitem{K2019}
H. Kr\"{o}ner,
A note on expansion of convex plane curves via inverse curvature flow.
\emph{NoDEA, Nonlinear Differ. Equ. Appl.} \textbf{26} (2019), Paper No. 9, 11 p.

\bibitem{K1987}
N.V. Krylov,
\emph{Nonlinear elliptic and parabolic equations of the second order.}
D. Reidel Publishing Co., Dordrecht, 1987.

\bibitem{K-L2021}
K.-K. Kwong and H. Lee, 
Higher order Wirtinger-type inequalities and sharp bounds for the isoperimetric deficit. 
\emph{Proc. Amer. Math. Soc.} \textbf{149} (2021), 4825--4840. 

\bibitem{K-W-W-W2022}
K.-K. Kwong, Y. Wei, G. Wheeler and V.-M. Wheeler,
On an inverse curvature flow in two-dimensional space forms.
\emph{Math. Ann.} \textbf{384} (2022), 285--308.


\bibitem{L-W2023}
M. Li and G. Wang,
\(\ell\)-convex Legendre curves and geometric inequalities.
\emph{Calc. Var. Partial Differ. Equ.} \textbf{62} (2023), Paper No. 135, 24 p.


\bibitem{M-C2014}
L. Ma and L. Cheng,
A non-local area preserving curve flow.
\emph{Geom. Dedicata} \textbf{171} (2014), 231--247.


\bibitem{M-P-W2013}
Y.Y. Mao, S.L. Pan and Y.L. Wang,
An area-preserving flow for closed convex plane curves. 
\emph{Int. J. Math.} \textbf{24} (2013), 1350029, 31p.



\bibitem{N-N2019}
T. Nagasawa and  K. Nakamura, 
Interpolation inequalities between the deviation of curvature and the isoperimetric ratio with applications to geometric flows.
\emph{Adv. Differ. Equ.} \textbf{24} (2019), 
581--608.


\bibitem{N2021}
K. Nakamura,
An application of interpolation inequalities between the deviation of curvature and the isoperimetric ratio to the length-preserving flow.
\emph{Discrete Contin. Dyn. Syst., Ser. S} \textbf{14} (2021), 1093--1102.



\bibitem{P-Y2008}
S.L. Pan and J.N. Yang,
On a non-local perimeter-preserving curve evolution problem for convex plane curves.
\emph{Manuscr. Math.} \textbf{127} (2008), 469--484.



\bibitem{P-Y2019}
S.L. Pan and Y.L. Yang,
An anisotropic area-preserving flow for convex plane curves. 
\emph{J. Differ. Equations} \textbf{266} (2019), 3764--3786.


\bibitem{P-Z2020}
S.L. Pan and Z.Y. Yang,
On a perimeter-preserving crystalline flow.
\emph{J. Differ. Equations} \textbf{269} (2020), 1944--1962.

\bibitem{S-T-W2020}
N. Sesum, D.-H. Tsai and X.L. Wang,
Evolution of locally convex closed curves in the area-preserving and length-preserving curvature flows,
\emph{Commun. Anal. Geom.} {\bf 28} (2020), 1863--1894.


\bibitem{S-X2019}
J. Scheuer and C. Xia,
Locally constrained inverse curvature flows.
\emph{Trans. Amer. Math. Soc.} \textbf{372} (2019), 6771--6803.


\bibitem{S2014}
R. Schneider,
\emph{Convex Bodies: The Brunn-Minkowski Theory},
Cambridge University Press, Cambridge, 2014.



\bibitem{T-W2015}
D.-H. Tsai and X.L. Wang,
On length-preserving and area-preserving nonlocal flow of convex closed plane curves. 
\emph{Calc. Var. Partial Differ. Equ.} \textbf{54}  (2015), 3603--3622.



\bibitem{U1990}
J. Urbas,
On the expansion of starshaped hypersurfaces by symmetric functions of their principal curvatures.
\emph{Math. Z.} \textbf{205} (1990), 355--372.


\bibitem{W2023}
X.L. Wang,
The evolution of area-preserving and length-preserving inverse curvature flows for immersed locally convex closed plane curves. 
\emph{J. Funct. Anal.} \textbf{284} (2023), 109744, 25p.


\bibitem{W-W-Y2018}
X.L. Wang, W.F. Wo and M. Yang,
Evolution of non-simple closed curves in the area-preserving curvature flow. 
\emph{Proc. R. Soc. Edinb., Sect. A, Math.} \textbf{148} (2018), 659--668 .



\bibitem{X2017-1}
C. Xia,
Inverse anisotropic mean curvature flow and a Minkowski type inequality.
\emph{Adv. Math.} \textbf{315} (2017), 102--129.



\bibitem{X2017-2}
C. Xia,
Inverse anisotropic curvature flow from convex hypersurfaces.
\emph{J. Geom. Anal.} \textbf{27} (2017), 2131--2154.



\bibitem{Y-Z-Z2023}
Y.L. Yang, Y.M. Zhao and Y.L. Zhang,
On an area-preserving locally constrained inverse curvature flow of convex curves.
\emph{Nonlinear Anal. Theory Methods Appl. Ser. A, Theory Methods} \textbf{230} (2023), 113245.





\end{thebibliography}

\end{document}